\documentclass[11pt,a4paper]{article}
\usepackage[utf8]{inputenc}

\usepackage{amssymb,amscd,amsfonts,amsbsy,amsmath,amsthm,mathrsfs}
\usepackage{dsfont}
\usepackage{enumerate,color}

\newtheorem{proposition}{Proposition}[section]
\newtheorem{theorem}[proposition]{Theorem}
\newtheorem{lemma}[proposition]{Lemma}

\theoremstyle{definition}

\theoremstyle{remark}
\newtheorem{remark}[proposition]{Remark}

\numberwithin{equation}{section}

\newcommand{\dd}{{\rm\,d}}
\newcommand{\dx}{{\rm\,d}x}
\newcommand{\ds}{{\rm\,d}s}

\newcommand{\R}{\mathbb{R}}

\newcommand{\Y}{\mathscr{Y}}
\newcommand{\N}{\mathbb N}
\newcommand{\tB}{\tilde{B}}
\newcommand{\PM}{\mathcal{P}\!\mathcal{M}}
\newcommand{\sgn}{{\rm sgn\,}}
\newcommand{\un}{\mathbf{1}\!\!{\rm I}}
\newcommand{\III}[1]{{\vert\kern-0.25ex\vert\kern-0.25ex\vert #1 
    \vert\kern-0.25ex\vert\kern-0.25ex\vert}}
\newcommand{\BB}[1]{{ [\![ #1 ]\!]}} 
\DeclareMathOperator{\esssup}{ess\,sup}

\begin{document}
     \baselineskip=13pt
\title{Sharp well-posedness and blowup results \\ for parabolic systems of the Keller--Segel type} 
\author{Piotr Biler$^1$, Alexandre Boritchev$^2$ and Lorenzo Brandolese$^2$\\ 
\small{ $^1$ Instytut Matematyczny, Uniwersytet Wroc{\l}awski,}\\ 
\small{ pl. Grunwaldzki 2, 50--384 Wroc{\l}aw, POLAND}\\ 
\small{ $^2$Université Claude Bernard Lyon 1,}\\
\small{CNRS UMR 5208, Institut Camille Jordan,}\\
\small{F-69622 Villeurbanne, FRANCE }\\
}

\date{\today}
\maketitle

\begin{abstract}
We study two toy models obtained after a slight modification of the nonlinearity of the usual doubly parabolic Keller--Segel system. For these toy models, both consisting of a system of two parabolic equations, we establish that for data which are, in a suitable sense, smaller than the diffusion parameter $\tau$ in the equation for the chemoattractant, we obtain global solutions, and for some data larger than $\tau$, a finite time blowup. In this way, we check  that our size condition for the global existence is sharp for large $\tau$, up to a logarithmic factor. 
\end{abstract}

\section{Introduction and main results  }

This paper is concerned with parabolic systems {\rm (TM)} and {\rm (TM')} below, depending on a diffusion parameter $\tau>0$: 
\begin{equation}
\left\{
\begin{aligned}
&u_t =\Delta u-u \Delta \varphi,\\
&\tau \varphi_t=\Delta \varphi+u,\\
&u(0)=u_0,\ \  \varphi(0)=\varphi_0,
\end{aligned}
\right.
\qquad x\in\R^d,\ t>0,
\tag{TM}
\end{equation}
\medskip 
and 
\begin{equation}
\left\{
\begin{aligned}
&u_t =\Delta u+(\Delta\varphi)^2,\\
&\tau \varphi_t=\Delta \varphi+u,\\
&u(0)=u_0,\ \  \varphi(0)=\varphi_0,
\end{aligned}
\right.
\qquad x\in\R^d,\ t>0,
\tag{TM'}
\end{equation}
The above systems degenerate into the quadratic nonlinear heat equation {\rm (NLH)} when $\tau=0$ and 
a compatibility condition, $\Delta \varphi_0=-u_0$, is put on $\varphi_0$:

\begin{equation}\tag{NLH}
\left\{\begin{aligned}
&u_t =\Delta u+u^2,\quad\\
&u(0)=u_0,\quad
\end{aligned}
\quad \qquad \qquad x\in\R^d,\ t>0.
\right.
\end{equation}

These model systems are introduced in order to show the influence of the parameter $\tau$ in the second equation (a linear nonhomogeneous heat equation) on the size of admissible initial data leading to global-in-time solutions, and to finite time blowup, respectively. 
Our main motivation is to analyze those issues for the Keller--Segel system {\rm (PP)}, continuing the analysis started in \cite{BBB1}. 
Here {\rm (PP)} denotes the doubly parabolic Keller--Segel system describing chemotaxis, given by 
\begin{equation}
\left\{
\begin{aligned}
&u_t =\Delta u-\nabla\cdot(u\nabla \varphi),\\
&\tau \varphi_t=\Delta \varphi+u,\\
&u(0)=u_0,\ \  \varphi(0)=\varphi_0,
\end{aligned}
\right.
\qquad x\in\R^d,\ t>0,
\tag{PP}
\end{equation}
where $\tau>0$, 
 and {\rm (PE)} is the parabolic-elliptic Keller--Segel system with $\tau=0$
\begin{equation}
\left\{
\begin{aligned}
&u_t =\Delta u-\nabla\cdot(u\nabla \varphi),\\
&\Delta \varphi+u=0,\\
&u(0)=u_0,   
\end{aligned}
\right.
\qquad x\in\R^d,\ t>0.
\tag{PE}
\end{equation}
The existence of solutions for {\rm (PP)} has been studied in the recent work \cite{BBB1}. 
The issue of blowup is largely open for {\rm (PP)}, except for  \cite{Wi1,Wi2} where  the radially symmetric problem is considered in a ball and in \cite{Wi3} in the whole space. Note that  some concentration phenomena for {\rm (PP)} have been shown in \cite{CCE}. There, a supplementary information on $L^1$ solutions is derived from entropy functionals and other specific properties of those drift-diffusion systems.

Note that {\rm (TM)} and {\rm (TM')}  have the same structure of steady states as {\rm (NLH)}, so they are related in a way similar to that as {\rm(PP)} relates to {\rm (PE)}. 
However, {\rm (NLH)} has a number of specific properties (such as variational structure, energy functional) that are not immediately extended to systems like {\rm (TM)} and {\rm (TM')}. 

All the considered systems feature cross-diffusion terms which make their analysis delicate. This justifies, in a sense, the application of Besov spaces in their analysis, cf. \cite{I,Lem} for well- and ill-posedness issues.

The analysis of the toy models {\rm (TM)} and {\rm (TM')} shed some light on the optimal spaces where it is reasonable to address the existence problem for (PP). Indeed, we prove in Sec. 2.2 existence for {\rm (TM)} in the framework of Besov-type spaces, very close to optimal spaces with respect to the admissible initial data leading to local-in-time solutions. Next, in Sec 3.1, a couple of such existence and regularity results is shown for {\rm (TM)}, and for {\rm (TM')} in Sec. 3.2, in the framework of pseudomeasures,  
similarly to the presentation in \cite[Sec. 2]{BBB1}. 
Section 4 is devoted to proofs of blowup, for {\rm (TM)} in Sec. 4.1 and for {\rm (TM')} in Sec. 4.2. 
The proofs are based on the idea going back to \cite{M-S} (and used for {\rm (PE)} in \cite{BB1}) where the Fourier transform of solutions is analyzed: we prove that it tends to infinity in an appropriate sense in finite time, so that the solution itself necessarily becomes non-smooth. 
Thus, the technique used in the proof of Theorem \ref{theoremCC} differs much from the usual approaches to nonlinear parabolic equations (see e.g. \cite[Ch. 17]{QS}). 
A more traditional approach involving moments of solutions is presented in Sec. 5 for the model {\rm (TM')} considered in bounded domains.


Besides formal similarities of {\rm (TM)}, {\rm (TM')} with {\rm (NLH)} and  {\rm (PE)}
when $\tau=0$, or with {\rm (PP)} when  $\tau>0$, there are rigorous results on singular perturbation limits $\tau\searrow 0$ for small solutions of Keller--Segel systems for which blowup does not occur. 
Indeed, 
solutions with $\varphi_0=(-\Delta)^{-1}u_0$ are recovered for $\tau=0$, in the sense of suitable convergences in \cite{Rac,BB2,Lem,K-O} where various functional settings are proposed. 
Similar results hold also for the toy models but our main interest here is rather the behaviour for large $\tau$ than for $\tau$ close to $0$. 

Note however, that a natural conjecture that blowup phenomena should be {\em continuous with respect to  the  parameter $\tau\to 0$} has not been, unfortunately, up to now  rigorously proved, neither for {\rm (TM)}$\to${\rm (NLH)} nor {\rm (PP)}$\to${\rm (PE)}. 
 \bigskip

In the global-in-time existence results below the initial data $u_0$ does not have to be positive.
\\
We make the assumption $\varphi_0=0$ which simplifies the presentation but it is not essential. 

Our analysis illustrate that existence and well-posedness results for both toy models have nearly optimal character if one considers (both the regularity  and) the size of the initial data. This in turn sheds new light on these issues for the classical parabolic-parabolic Keller--Segel system, as discussed in \cite{BBB1}.


\bigskip

\subsection*{Notation}


In this paper we adopt the following notation and conventions. The expression $A\lesssim B$, where $A$ and $B$ may depend on several parameters, means that there exists a constant $c>0$, depending only on the space dimension, such that $A\le c\,B$. When both $A\lesssim B$ and $B\lesssim A$
we will write $A\approx B$. 

For a function $f\in L^1(\R^d)$, 
the definition of the Fourier transform that we use is 
$\widehat f(\xi)=\int f(x)\exp(-i\xi\cdot x)\dd x$. This definition is extended
to $\mathscr{S}'(\R^n)$, the space of tempered distributions, in the usual way.
The space of general distributions is denoted $\mathscr{D}'(\R^d)$.

%
%
%

In this paper we will deal with mild solutions. 
These are solutions of the integral formulation of (PP). 
The exact meaning of the integral must be understood in the specific functional setting. 


\section{Besov spaces}

In this section, we study the local and the global solvability of (TM) in
Besov-type spaces.
The singularity of $\Delta\varphi$ is somehow too strong to directly perform
the bilinear estimates arising from the the nonlinear term $u\Delta\varphi$ in the usual way.
For this reason, we will study (TM) as a perturbation of the original model {\rm (PP)}. 
Accordingly, the nonlinear term in will be decomposed as 
\begin{equation}
\label{eq:pertu}
u\Delta\varphi=\nabla \cdot(u\nabla \varphi)-\nabla u \cdot \nabla \varphi.
\end{equation}
The first term in the right-hand is nothing but the nonlinearity of (PP).
In our previous paper \cite{BBB1}, the solutions to (PP) were constructed in a ball of the space
\begin{equation} \label{E_p}
{\mathcal E}_p:=\left\{ u \in L^{\infty}(0,\infty;L^p(\R^d)),\ \III{u}_p:=\esssup t^{1-d/(2p)} \Vert u(t) \Vert_p <\infty\ \right\}.
\end{equation}
But the presence of the second term requires to have additional information on
$\|\nabla u(t)\|_p$. For this reason, 
we will need to work in the subset of the space ${\mathcal E}_p$ given by 
\begin{equation} \label{F_p}
{\mathcal F}_p:=\left\{ u \in  L^{\infty}(0,\infty;W^{1,p}(\R^d)):\ 
\BB{u}_p:= \III{u}_{p}+\III{u}_{1,p}<\infty \right\},
\end{equation}
where we define
$$
\III{u}_{1,p}:=\esssup_{t>0} t^{3/2-d/(2p)} \Vert \nabla u(t) \Vert_p.
$$

The purpose of this section is to establish the following:

\begin{theorem} \label{BesT1'}
Let $d\ge3$.
For $2d/3<p<d$ and $2/d-1/p<1/q<1/d$, there exist $C_{q,d},C'_{q,d}>0$ such that if we denote
\begin{equation} \label{Besgamma}
\gamma:=1/2-d/2(1/p-1/q)
\end{equation}
(when $p \rightarrow d$, we have $q \rightarrow d$, which forces $\gamma \rightarrow 1/2$), and
\begin{equation} \label{BesIn'}
\Vert u_0 \Vert_{\dot{B}_{p,\infty}^{-(2-d/p)}} < C_{q,d} \tau^{\gamma},
\end{equation}
then there exists a global mild solution $u \in {\mathcal F}_p$ to {\rm (TM)}, such that
\begin{equation}
\label{balluni}
\BB{u}_p < C'_{q,d} \Vert u_0 \Vert_{\dot{B}_{p,\infty}^{-(2-d/p)}}.
\end{equation}
Such a solution is uniquely defined by condition~\eqref{balluni}.
Moreover, $u(t)-{\rm e}^{t\Delta}u_0\in BC(0,\infty;L^{d/2})$.

\end{theorem}


We will reduce the proof of this theorem
to a few lemmas, established in the following subsections. 
For larger initial data the size condition~\eqref{BesIn'} may not apply. In this case,
local-in-time solutions can be proved to exist, under stronger regularity conditions 
on the initial data.

\begin{remark}
As mentioned in the introduction, our main interest is for $\tau\gg1$.
Notice that the exponent in~\eqref{Besgamma} is such that $\gamma<1/2$. 
Thus, according to Theorem~\ref{BesT1'}, the size of the admissible data for the global existence would be worse than $O(\sqrt\tau)$, as $\tau\to+\infty$. This is not completely satisfactory, in view of next blowup result (Theorem~\ref{theoremCC} below): the latter will make evidence
of a class of initial data of size $O(\tau)$, whose solutions
blow up in finite time. Thus, there is a substantial gap between our
two results on global existence and on finite time blowup.

Such a gap is due to a technical limitation, related to the use of Besov spaces. Assuming that the initial data belong to slightly less rough spaces, we can almost completeley close this gap, up to a logarithmic factor in $\tau$.
This will motivate the analysis of (TM) in a different functional
setting. See Section~\ref{sec:PM}.
\end{remark}

%
%
%

\subsection{Review of known estimates in $\mathcal{E}_p$}

In this short subsection we briefly recall without proof some results obtained in 
\cite{BBB1} in order to study the original system {\rm (PP)}.

We recall the classical $L^p$-$L^q$ inequalities for the heat semigroup (see e.g. \cite[Ch. 15, (1.15)]{T3} or   \cite{G-M}), 
valid for $1 \le p \le q \le \infty:$   
\begin{equation} 
\label{heat}
\begin{split}
\Vert {\rm e}^{t\Delta} f \Vert_q &\le C(d,p,q) t^{-d(1/p-1/q)/2} \Vert f \Vert_p,\\ \Vert \nabla {\rm e}^{t\Delta} f \Vert_q &\le C(d,p,q) t^{-1/2-d(1/p-1/q)/2} \Vert f \Vert_p.
\end{split}
\end{equation}

Following~\cite{BBB1}, let us introduce the (respectively, linear and bilinear) operators $L$ and $B$ given by:
\begin{align}
\label{L}
&Lz(t):=\tau^{-1} \int_{0}^{t}{\nabla {\rm e}^{\tau^{-1} (t-s) \Delta} z(s) \ds},\ 
\\ \label{B}
&B(u,z)(t):=-\int_{0}^{t}{\nabla {\rm e}^{(t-s) \Delta} \cdot(u(s) Lz(s)) \ds}.
\end{align}

Below, all constants are implicitly assumed to depend on $d$. They are also implicitly assumed to depend on $p$ and $q$ in the lemmas. 

\begin{lemma} \label{BesL1}
If $0 \le 1/q \le 1/p < 1/q+1/d$ and $0<1/p$, then for $t>0$ 
$$
\Vert Lz(t) \Vert_q \le C \tau^{-1/2+d/2(1/p-1/q)} t^{-1/2+d/2q} \III{z}_p.
$$
\end{lemma}

\begin{lemma} \label{BesL2}
If $1/d < 1/p+1/q \le 1$, $0 \le 1/q < 1/d$ and the conditions of Lemma~\ref{BesL1} hold, then we have
$$
\III{B(u,z)}_p \le C \tau^{-1/2+d/2(1/p-1/q)} \III{u}_p \III{z}_p.
$$
\end{lemma}
In order to apply both lemmas, we should check the compatibility of the conditions of the exponents~$p$ and~$q$. The two lemmas, combined, require that:
$$
|1/p-1/d|<1/q \le \min\{1/p,1-1/p\};\quad 1/q<1/d\quad\text{and}\quad 0<1/p,
\quad\text{with}\quad d>1.
$$
This requires choosing $p$ so that $1/d<2/p<\min\{1+1/d,4/d\}$.
In other words, in order to find a suitable $q$ to use the lemmas above, a necessary and sufficient assumption is:
$$
d \ge 2;\quad \max\{d/2,\ 2d/(d+1)\}<p<2d.
$$ 
%
The last lemma in~\cite{BBB1} that we will need is the following:
\begin{lemma} \label{BesL4}
If $d \ge 2$ and $2d/3<p\le d$ and $u \in {\mathcal E}_p$, we have:
$$
\Vert B(u,u)(t) \Vert_{d/2} \le C \tau^{-3/2+d/p} \III{u}_p^2.
$$
\end{lemma}

\subsection{New estimates in $\mathcal{F}_p$}

Here we complete the results of the previous subsection with some additional
estimates that we need to solve (TM).
We will assume that $d \ge 3$. Otherwise,  we cannot find admissible parameters for the technical lemmas which we are using.
In the same way as we did when studying the classical Keller--Segel model in \cite[Sec. 3]{BBB1}, all constants are implicitly assumed to depend on $d$, 
as well as on $p$ and $q$ in the lemmas.

The second term in the right-hand side of~\eqref{eq:pertu} motivates the introduction
of a second bilinear operator. 
We define
\begin{equation} \label{Btilde}
\tB(u,z)(t):=\int_{0}^{t}{ {\rm e}^{(t-s) \Delta} (\nabla u(s)\cdot Lz(s)) \ds}.
\end{equation}
We will proceed similarly to the proofs in \cite{BBB1} of the results in the previous subsection, where we have already estimated $\III{B(u,z)}_p$. Here we need to study the quantities 
$\III{B(u,z)(t)}_{1,p}+\III{\tB(u,z)(t)}_{1,p}$ and $\III{\tB(u,z)(t)}_p$.
This will be achieved, respectively in Lemma~\ref{BesL2'} and Lemma~\ref{BesL2''} below.

Both of these lemmas require a preliminary linear estimates for
$\III{L \nabla v(t)}_{q'}$.

\begin{lemma} \label{BesL1'}
Under the assumptions
$$
0 \le 1/q' \le 1/p < 1/q'+1/d;\quad p<d,
$$
we have $\III{L \nabla v(t)}_{q'} \le C \tau^{-1/2+d/2(1/p-1/q')} t^{-1+d/2q'} \III{v}_{1,p}$.
\end{lemma}

\begin{proof}
Using \eqref{heat} we get:
\begin{align*}
 \III{L \nabla v(t)}_{q'} &\le C \tau^{-1} \int_{0}^{t}{[\tau^{-1} (t-s)]^{-1/2-d/2(1/p-1/q')} \Vert \nabla v(s) \Vert_p \ds}
\\
&\le C \tau^{-1/2+d/2(1/p-1/q')} \III{v}_{1,p} \int_{0}^{t}{(t-s)^{-1/2-d/2(1/p-1/q')} s^{-3/2+d/2p} \ds}
\\
&\le C \tau^{-1/2+d/2(1/p-1/q')} t^{-1+d/2q'} \III{v}_{1,p}.
\end{align*}
\end{proof}

\begin{lemma} \label{BesL2'}
If $q'$ satisfies the assumptions of Lemma~\ref{BesL1'} (and \textit{a~fortiori} those of Lemma~\ref{BesL1} with $q'$ playing the role of $q$), and moreover
$$
2/d<1/p+1/q' \le 1;\quad 0 \le 1/q' < 1/d,
$$
then
$$
\Vert \nabla B(u,z) (t) \Vert_p + \Vert \nabla \tB(u,z) (t) \Vert_p \le C \tau^{-1/2+d/2(1/p-1/q')} t^{-3/2+d/2p} \BB{u}_{p} \BB{z}_{p}.
$$
\end{lemma}

\begin{proof}
Since we have
$$
\nabla B(u,z)=B(\nabla u, z)+B(u,\nabla z)\quad\text{and}\quad \nabla \tB (u,z)=-B(\nabla u,z),
$$
it suffices to estimate $\Vert B(\nabla u, z) \Vert_p+\Vert B(u,\nabla z)\Vert_p$.
\\
First, we use \eqref{heat} and H{\"o}lder's inequality and then, respectively, Lemma~\ref{BesL1} and Lemma~\ref{BesL1'}), ending by the change of variables $\tilde{s}=t/s$:
\begin{align*}
 \Vert B(\nabla u, z) \Vert_p &\le C \int_{0}^{t}{(t-s)^{-1/2-d/2q'} \Vert \nabla u(s) \Vert_p \Vert Lz(s) \Vert_{q'} \ds}
\\
& \le C \tau^{-1/2+d/2(1/p-1/q')}  \int_{0}^{t}{(t-s)^{-1/2-d/2q'} s^{-3/2+d/2p} s^{-1/2+d/2q'} \ds} \cdot \III{u}_{1,p} \III{z}_{p}
\\
& \le C \tau^{-1/2+d/2(1/p-1/q')} t^{-3/2+d/2p} \III{u}_{1,p} \III{z}_{p}.
\end{align*}
Then, similarly, we get:
\begin{align*}
 \Vert B(u,\nabla z)\Vert_p &\le C \int_{0}^{t}{(t-s)^{-1/2-d/2q'} \Vert u(s) \Vert_p \Vert L \nabla z(s) \Vert_{q'} \ds}
\\
& \le C \tau^{-1/2+d/2(1/p-1/q')} \int_{0}^{t}{(t-s)^{-1/2-d/2q'} s^{-1+d/2p} s^{-1+d/2q'} \ds} \cdot \III{u}_{p} \III{z}_{1,p}
\\
& \le C \tau^{-1/2+d/2(1/p-1/q')} t^{-3/2+d/2p} \III{u}_{p} \III{z}_{1,p}.
\end{align*}
\end{proof}

\begin{lemma} \label{BesL2''}
With the same assumptions as in Lemma~\ref{BesL1} for $q''$ playing the role of $q$, and if moreover
$$
2/d<1/p+1/q'' \le 1;\quad 0 \le 1/q''<2/d,
$$
we have
$$
\Vert \tB(u,z) (t) \Vert_p \le C \tau^{-1/2+d/2(1/p-1/q'')} t^{-1+d/2p} \BB{u}_{p} \BB{z}_{p}.
$$
\end{lemma}

\begin{proof}
Using \eqref{heat}, H{\"o}lder's inequality and then Lemma~\ref{BesL1} we get 
\begin{align*}
 \Vert \tB(u,z) (t) \Vert_p 
 &\le C \int_{0}^{t}{(t-s)^{-d/2q''} \Vert \nabla u(s) \Vert_p \Vert Lz(s) \Vert_{q''} \ds}
\\
&  \le C \tau^{-1/2+d/2(1/p-1/q'')} \int_{0}^{t}{(t-s)^{-d/2q''} s^{-3/2+d/2p} s^{-1/2+d/2q''} \ds} \III{u}_{1,p} \III{z}_{p}
\\
& \le C \tau^{-1/2+d/2(1/p-1/q'')} t^{-1+d/2p} \III{u}_{1,p} \III{z}_{p}.
\end{align*}\end{proof}

To find an admissible $p<d$, so that $q$ from Lemma~\ref{BesL2}, $q'$ from Lemma~\ref{BesL2'} and $q''$ from Lemma~\ref{BesL2''} exist, a sufficient condition is
\begin{equation} \label{Besrange}
\max\{1/p-1/d,2/d-1/p\} < 1/q,1/q',1/q'' <\min\{1/d,1-1/p\}.
\end{equation}
(The conditions on $q$ and $q''$ are in fact less restrictive than those needed for $q'$).
This is possible for $d\ge3$ and $d/2<p<d$.

\medskip

Our last lemma will be useful to see that the fluctuation of the solution, i.e. the difference
$u-{\rm e}^{\cdot\Delta}u_0$, not only belongs to the space $\mathcal{F}_p$ where $u$ will be constructed,
but also to $L^\infty(0,\infty;L^{d/2})$.

\begin{lemma} \label{BesCont}
Let $d \ge 3$ and $d/2<p<d$.
There exists a constant $C_{\tau}>0$ such that 
\begin{align} \label{Besa}
& \Vert \tB(u,u) \Vert_{d/2} \le C_{\tau} \BB{u}_p^2.
\end{align}
\end{lemma}

\begin{proof}
Here and below we assume that $d/2<p<d$, and $r$, $q=q'=q''$ given by
$$
1/r=3/(2d)+1/(2p);\quad 1/p+1/q=1/r.
$$
Therefore
$$
2/d-1/p < 1/q=3/(2d)-1/(2p) < 1/d \le 1-1/p,
$$
and moreover
$$
1/q>1/p-1/d,
$$
so the assumptions \eqref{Besrange} are satisfied.   
\medskip
In the four inequalities below, we use respectively \eqref{heat}, H{\"o}lder's inequality, Lemma~\ref{BesL1} and the change of variables $\tilde{s}=t/s$:
\begin{align*}
 \Vert \tB(u,u)(t) \Vert_{d/2} &\le C \int_{0}^{t}{(t-s)^{-d/2(1/r-2/d)} \Vert \nabla u(s) Lu(s) \Vert_{r} \ds} 
\\
& \le C \int_{0}^{t}{(t-s)^{1/4-d/(4p)} \Vert \nabla u(s) \Vert_p \Vert Lu(s) \Vert_q \ds}
\\
& \le C_{\tau} \int_{0}^{t}{(t-s)^{1/4-d/(4p)} s^{-5/4+d/(4p)} \ds} \cdot \BB{u}_p^2\\
& \le C_{\tau} \BB{u}_p^2.
\end{align*}
\end{proof}

%

\begin{proof}[Proof of Theorem~\ref{BesT1'}]
To prove the first theorem, we observe that if 
$u_0 \in \dot{B}_{p,\infty}^{-(2-d/p)}$, then 
$\nabla u_0\in \dot{B}_{p,\infty}^{-2(3/2-d/2p)}$, and the Besov norm of $\nabla u_0$ is equivalent to that of $u_0$.
Applying twice \cite[Theorem 2.34]{BCD-Book}, first to $u_0$, next to $\nabla u_0$, 
we deduce that 
${\rm e}^{\cdot \Delta} u_0 \in {\mathcal F}_p$
and 
\[
\|{\rm e}^{\cdot \Delta} u_0 \|_{\mathcal{F}_p} \lesssim 
\|u_0\|_{\dot B_{p,\infty}^{-(2-d/p)}}.\]
The usual fixed point lemma (see, e.g., \cite[Theorem 13.2]{Lem-Book1} or \cite[Lemma 1.1.1]{B-book}) applies in a ball of $\mathcal{F}_p$,
thanks to the bilinear estimate
\[
\BB{B(u,v)}_{p}+\BB{\tilde B(u,v)}_{p}\le C\tau^{-1/2+(d/2)(1/p-1/q)}
\BB{u}_p\BB{v}_p,
\]
that follows combining Lemmas~\ref{BesL1'},~\ref{BesL2'} and~\ref{BesL2''}
with $q=q'=q''$.
The solution $u$ is thus constructed in a ball of $\mathcal{F}_p$.

The fluctuation
$u-{\rm e}^{\cdot\Delta}u_0$ also belongs 
to $BC(0,\infty;L^{d/2})$. Indeed, for $t>0$, we have the estimate,
\[
\|u(t)-e^{t\Delta}u_0\|_{d/2}\le
\|B(u,u)(t)\|_{d/2} +
 \| \tB(u,u)(t)\|_{d/2} \le C_{\tau} \BB{u}_p^2,
\]
that follows combining Lemma~\ref{BesL4} and Lemma~\ref{BesCont}.
The continuity with respect to $t$ is standard and we skip it here (see \cite{BBB1}).
\end{proof}

\begin{remark}
\label{rem:local}
For more regular, but possibly large, initial data, local-in-time existence of solutions to (TM) can be proved in several ways.
For example, if $d\ge3$, $u_0\in L^{d/2}(\R^d)$ and $p$ is as in Theorem~\ref{BesT1'},
then there exists $T>0$ such that a solution $u\in \widetilde{\mathcal{F}}_{p,T}$ does exist.
Here, $\mathcal{F}_{p,T}$ is defined in the same way as $\mathcal{F}_p$, with the time interval $(0,T)$ replacing $(0,\infty)$. The space $\widetilde{\mathcal{F}}_{p,T}$ is the subspace of $\mathcal{F}_{p,T}$ of functions $v$ satisfying the additional condition
\[
\lim_{t\to0^+}\Bigl(t^{1-d/(2p)}\|v(t)\|_p+t^{3/2-d/(2p)}\|\nabla v(t)\|_p\Bigr)=0.
\]
A simple approximation argument (using the density of $L^p$ in $L^{d/2}$ and the
$L^{d/2}$-$L^p$ estimates for the heat hernel and its gradient) proves that if $u_0\in L^{d/2}(\R^d)$, then ${\rm e}^{t\Delta}u_0\in \widetilde{\mathcal{F}}_{p,T}$.
Thus, the size condition on the data needed to apply the usual fixed point theorem 
can be ensured just taking $T>0$ small enough (with a nontrivial dependence on $u_0$ and not just $|u_0|_{d/2}$, due to the approximation procedure). The solution that we obtain is easily proved to be in $BC(0,T;L^{d/2})$ and to be unique in $\widetilde{\mathcal{F}}_{p,T}$.
\end{remark}

\begin{remark}
An analogous analysis of (TM') is not possible since the bilinear term has a more singular stricture when using Besov spaces. This is one of the reasons why we need to introduce pseudomeasures.
\end{remark}

\section{{Pseudomeasures}}
\label{sec:PM}
 
In this section we will study both toy models (TM) and (TM') in pseudomeasure spaces.
We begin by recalling some definitions and statements of results from \cite{BBB1}.
The functional setting is more restrictive than in the previous section, but the size conditions will be less stringent when $\tau\gg1$ and the proofs will be shorter. Indeed, when working in pseudomeasure spaces, it is no longer necessary to treat the nonlinearity of the toy models as a perturbation to that of (PP).

Let $a\in\R$. We introduce the pseudomeasure space
\begin{equation}
\label{PMa}
\PM^a=
\{f\in \mathscr{S}'(\R^d)\colon \|f\|_{\PM^a}=\esssup_{\xi\in\R^d}|\xi|^a|\widehat f(\xi)|<\infty\},
\end{equation}
where $\widehat f$ denotes the Fourier transform of the tempered distribution $f$.

We will construct our solutions in the space
\begin{equation}
\Y_a=\{u\in L^\infty_{\rm loc}(0,\infty;\mathscr{S}'(\R^d))\colon
\|u\|_{\Y_a}=\esssup_{t>0,\,\xi\in\R^d}t^{1+(a-d)/2}|\xi|^a|\widehat u(\xi,t)|<\infty\}.
\end{equation}
When $a=d-2$, the space $\Y_{d-2}$, agrees with the space 
\[
\mathcal{X}=L^\infty(0,\infty;\PM^{d-2}),
\]
already used in \cite[Theorem 2.1]{BCGK} to establish a global existence  result
for the parabolic-elliptic Keller--Segel system {\rm (PE)} for $d\ge4$
and small initial data in $\PM^{d-2}$.
When $a\not=d-2$, our space $\Y_a$ is slightly larger than the space 
\[
\mathcal{Y}_a:=\Y_a\cap\mathcal{X}
\]
 considered in \cite[Section 4]{BCGK}  $(d\ge3)$
or in \cite{Rac} $(d=2)$.

We start recalling a result of~\cite{BBB1} for the classical parabolic-parabolic Keller--Segel system:
\begin{theorem}
\label{th:PM}
Let $d\ge2$, $\tau>0$ and $u_0\in \PM^{d-2}(\R^d)$.
There exists $\kappa_d>0$, depending only on $d$, such that if
\begin{equation}
\label{small-assu}
\begin{split}
&\|u_0\|_{\PM^{d-2}} <\kappa_d \max\Bigl\{1, \frac{3^3\,\tau}{({\rm e}^{}\ln\tau)^3}\Bigr\}
\end{split}
\end{equation}
then  {\rm(PP)} possesses a global mild solution~$u\in \mathcal{X}$.
More precisely:
\begin{itemize}
\item[-]
 If $\|u_0\|_{\PM^{d-2}} <\kappa_d$ then $u\in \mathcal{X}\cap \Y_{d-\frac43}$, and this solution is unique in the ball $\{u\colon \|u\|_{\Y_{d-\frac43}}< 2\kappa_d\}$.
\item[-]
If $\tau\ge {\rm e}^3$, then under the weaker condition $\|u_0\|_{\PM^{d-2}}< {3^3\kappa_d\,\tau}/{({\rm e}^{}\ln\tau)^3}$
we have $u\in \mathcal{X}\cap\Y_{d-4/\ln\tau}$, and $u$ is the unique solution in the ball 
$\{u\colon \|u\|_{\Y_{d-4/\ln\tau}}<2\kappa_d{3^3\,\tau}/{({\rm e}^{}\ln\tau)^3}\}$.
\end{itemize}
\end{theorem}

The analog of Theorem~\ref{th:PM} for our two toy models is the following result:

\begin{theorem}
\label{th:TM}
For $d\ge3$ the assertion of Theorem~\ref{th:PM} holds for both systems {\rm (TM)}
and {\rm (TM')}.
\end{theorem}

The proof of Theorem~\ref{th:TM} is conceptually the same as that of Theorem~\ref{th:PM}, but different technical restrictions appear on the
parameters that are involved in the estimates.
For reader's convenience we briefly outline the proof of Theorem~\ref{th:TM}.
This will allow us to illustrate why the case $d=2$ should be excluded
for the toy model.

\begin{lemma}
\label{lem:convo}
Let $0<\alpha,\beta<d$ such that $\alpha+\beta>d$.
Then
\[
|x|^{-\alpha}*|x|^{-\beta}=C(\alpha,\beta,d)|x|^{-(\alpha+\beta)+d},
\]
with
\begin{equation}
\label{eq:CC}
C(\alpha,\beta,d)=\pi^{d/2}
\frac{\Gamma(\frac{d-\alpha}{2})\Gamma(\frac{d-\beta}{2})\Gamma(\frac{\alpha+\beta-d}{2})}{\Gamma(\frac{\alpha}{2})\Gamma(\frac{\beta}{2})\Gamma({d-\frac{\alpha+\beta}{2})}}.
\end{equation}
\end{lemma}
\begin{proof}
See \cite[Lemma 2.1]{BCGK}.
\end{proof}

Another Lemma used in~\cite{BBB1} (refining \cite[Lemma 3.2]{Rac}), is the following:
\begin{lemma}
\label{lem:integ}
Let $s>0$, $A>0$, $\delta>0$, $0\le b\le1$ and $\delta_*=\min\{\delta,1\}$.
Then 
\[
\int_0^s {\rm e}^{-(s-\sigma)A}\sigma^{-1+\delta}\dd s\le 4\delta_*^{-1}A^{-b}s^{\delta-b}.
\]
\end{lemma}

\subsection{Pseudomeasures for the system {\rm (TM)}}

Taking the Fourier transform in {\rm (TM)} we get
\begin{equation}\label{FTM}
\begin{split}
\widehat u(\xi,t)
&={\rm e}^{-t|\xi|^2}\widehat u_0(\xi)
-\int_0^t {\rm e}^{-(t-s)|\xi|^2}\widehat{u\Delta \varphi}(\xi,s)\dd s
\\
&={\rm e}^{-t|\xi|^2}\widehat u_0(\xi)
  +(2\pi)^{-d}\int_0^t\!\!\int_{\R^d}{\rm e}^{-(t-s)|\xi|^2}\widehat{u}(\xi-\eta,s)|\eta|^2\widehat\varphi(\eta,s)\dd\eta\dd s.\\
\end{split}
\end{equation} 
Notice that
\[
\widehat \varphi(\eta,s)=\tau^{-1} \int_0^s {\rm e}^{-\tau^{-1}(s-\sigma)|\eta|^2}\widehat
u(\eta,\sigma)\dd \sigma.
\]
Therefore,
\begin{equation}
\label{FTMsol}
\begin{split}
&\widehat u(\xi,t)
={\rm e}^{-t|\xi|^2}\widehat u_0(\xi)\\
&\qquad\qquad
 +(2\pi)^{-d}\int_0^t\!\!\int_0^s\!\!\int_{\R^d} 
\frac{|\eta|^2}{\tau} 
{\rm e}^{-(t-s)|\xi|^2}{\rm e}^{-\frac{1}{\tau}(s-\sigma)|\eta|^2}
\widehat u(\xi-\eta,s)\widehat u(\eta,\sigma)\dd\eta\dd \sigma\dd s.
\end{split}
\end{equation}

By definition, by a solution  to {\rm (TM)} on $(0,T)$, with $T>0$, we mean  any function $u$ on $(0,\infty)$ with values in 
$\mathscr{S}'(\R^d)$ such that for a.e. $0<t<T$, 
$\widehat u(t,\cdot)$ is locally integrable in $\R^d$ and
for a.e. $0<s<t<T$ and $\xi\in\R^d$ the integrand in~\eqref{FTMsol}
is integrable in $(0,t)\times(0,s)\times \R^d$,
and $\eqref{FTMsol}$ holds for a.e. $(\xi,t)\in \R^d\times(0,T)$.

Such solutions can be constructed via the standard fixed point
algorithm as the limit $u=\lim_{k\to\infty} u_k$, where
\[
u_{k+1}={\rm e}^{t\Delta}u_0-\int_0^t {\rm e}^{(t-s)\Delta}\frac{1}{\tau}\biggl[u_k(s)\Delta\int_0^s {\rm e}^{\frac{1}{\tau}(s-\sigma)\Delta}u_k(\sigma)\dd \sigma\biggr]\dd s,
\] 
for $k=1,2,\dots\ $, 
in an appropriate function space. We already gave an instance of such a  construction for (PP), using pseudomeasure spaces in \cite[Sec. 2]{BBB1}.

\begin{proof}[Proof of Theorem~\ref{th:TM}. The case of (TM)]
We can write the problem {\rm (TM)}
in the integral form
\begin{equation*}
u(t)={\rm e}^{t\Delta}u_0+B'(u,u)(t),
\end{equation*}
where $B'$ is the bilinear operator
\begin{equation}
\label{eq:bilftm}
\widehat{B'(u,v)}(\xi,t)=
(2\pi)^{-d}\int_0^t\!\!\int_0^s\!\!\int_{\R^d}
\frac{|\eta|^2}{\tau}
{\rm e}^{-(t-s)|\xi|^2}{\rm e}^{-\frac{1}{\tau}(s-\sigma)|\eta|^2}
\widehat u(\xi-\eta,s)\widehat v(\eta,\sigma)\dd\eta\dd \sigma\dd s.
\end{equation}
Proceeding as in \cite{BBB1}, assuming for simplicity that $\|u\|_{\Y_a}=\|v\|_{\Y_a}=1$ and applying Lemma~\ref{lem:integ}, Lemma~\ref{lem:convo} and then again Lemma~\ref{lem:integ} with $\gamma=-b+1+(d-a)/2$ playing the role of $b$ we get the estimate:
\begin{align*}
 |\widehat{B'(u,v)}|(\xi,t)&\le (2\pi)^{-d} \tau^{-1} \int_0^t\!\!\int_0^s\!\!\int_{\R^d} {\rm e}^{-(t-s)|\xi|^2}{\rm e}^{-\frac{1}{\tau}(s-\sigma)|\eta|^2} s^{-1+(d-a)/2} \sigma^{-1+(d-a)/2} |\xi-\eta|^{-a} |\eta|^{-a+2}   \dd\eta\dd \sigma\dd s
\\
& \le \frac{8\tau^{b-1}}{(2\pi)^{d} (d-a)_*} \int_0^t\!\!\int_{\R^d} {\rm e}^{-(t-s)|\xi|^2} s^{-1+d-a-b}
|\xi-\eta|^{-a} |\eta|^{-a+2-2b} \dd\eta \dd s
\\
& \le \frac{8\tau^{b-1}C(a,a-2+2b,d)}{(2\pi)^{d} (d-a)_*} \int_0^t {\rm e}^{-(t-s)|\xi|^2} s^{-1+d-a-b}
|\xi|^{-2a+2-2b+d} \dd s
\\
& \le K'\tau^{b-1}|\xi|^{-a}t^{-1-(a-d)/2},
\end{align*}
where,
\begin{equation}
\label{eq:KK'}
K'=K'(a,b,d)=
 \frac{32\, C(a,a-2+2b,d)}{(2\pi)^d(d-a)_*(d-a-b)_*}.
\end{equation} 
However, the conditions for the validity of this estimate are not the same as for
the analogous estimate established for (PP) in~\cite{BBB1},
because we needed to make a different choice for the parameters. Namely, we need here
$a<d$ and $0\le b\le 1$ in the first application of Lemma~\ref{lem:integ}. We need also
$0<a<d$, $0<a+2b-2<d$ and $2a+2b-2>d$ for the application of Lemma~\ref{lem:convo}. And we
finally need $d-a-b>0$ and $0\le\gamma\le1$, with $\gamma= -b+1+(d-a)/2$, for the second application of Lemma~\ref{lem:integ}.
All these conditions are not compatible when $d=2$.
When $d\ge3$ they reduce to:
\begin{equation} \label{tmconda}
 d=3\colon \qquad 3-2b\le a,\qquad \textstyle\frac52-b<a<3-b, \qquad 0< b\le 1
\end{equation}
and
\begin{equation} \label{tmcondb}
d\ge4\colon\qquad
d-2b\le a<d-b, \qquad a\not=2,\qquad 0<b\le 1.
\end{equation}
For $d=3$ and $\frac32<a<3$, or for $d\ge4$, $d-2\le a<d$ and $a\not=2$, one can always find $b$ satisfying such conditions. Hence the bilinear operator $B'\colon\Y_a\times \Y_a\to\Y_a$ is continuous when
$a$ is in such ranges.
In this range,  $B'$ is continuous also as an operator $B'\colon\Y_a\times\Y_a\to\mathcal{X}$.
Indeed, to see this we
just need to change the choice of the parameter $\gamma$, and to take $\gamma=d-a-b$: we should now replace the previous condition $0\le-b+1+(d-a)/2\le1$ with the new condition $0\le d-a-b\le1$. But this new condition is satisfied given~\eqref{tmcondb}.

Now, for any $0<b\le1$, we can always choose, for example, $a=d-\frac43b$ in a such way that the required conditions on $a$ hold.
Moreover, $K'(d-\frac43b,b,d)\approx b^{-3}$ as $b\searrow 0$ and $K'(d-\frac43b,b,d)$ remains bounded
as $0<b\le1$ with $b$ bounded away from zero.
So, the following esimate 
holds:
\begin{equation}
\label{eq:bilp'}
\|B'(u,v)\|_{\Y_{d-\frac43b}}\le \frac{1}{4\kappa_d'}\, b^{-3}\tau^{b-1} \|u\|_{\Y_{d-\frac43b}}\|v\|_{\Y_{d-\frac43b}}
\qquad(0<b\le 1),
\end{equation}
for some constant $\kappa_d'>0$ only depending on~$d\ge3$.
For the system (TM), the remainder of the proof is now carried just like in Theorem~\ref{th:PM}.

\end{proof}

\begin{remark}
We do not know if the above result for {\rm (TM)} holds for $d=2$.
\end{remark}

\subsection{Pseudomeasures for the system {\rm (TM')}}

The goal of this subsection is to establish the proof, for the system (TM'), of Theorem~\ref{th:TM}.

\begin{proof}[Proof of Theorem~\ref{th:TM}. The case of (TM')]

The bilinear operator associated with {\rm (TM')} is

\begin{equation}
\label{bilTM'}
B''(u,v)(t)=\int_0^t {\rm e}^{(t-s)\Delta}(\Delta\varphi)(\Delta\psi)(s)\dd s 
\quad\text{with}\quad
\left\{
\begin{aligned}
&\varphi(s)=\frac{1}{\tau}\int_0^s{\rm e}^{(s-\sigma)\tau^{-1}\Delta}u(s)\dd s,\\
&\psi(s)=\frac{1}{\tau}\int_0^s{\rm e}^{(s-\lambda)\tau^{-1}\Delta}v(\lambda)\dd \lambda.
\end{aligned}
\right.
\end{equation}  
It is convenient to rewrite $B''(u,v)$ in terms of its Fourier transform,

\begin{equation*}
\widehat{B''(u,v)}(\xi,t)=
\frac{1}{(2\pi)^d}\int_0^t {\rm e}^{-(t-s)|\xi|^2}(|\cdot|^2\widehat\varphi)*(|\cdot|^2\widehat\psi)(\xi,s)\dd s,
\end{equation*}
or, more explicitly,  
\begin{equation}
\label{bilfTM'}
\begin{split}
\widehat{B''(u,v)}(\xi,t)
= \frac{\tau^{-2}}{(2\pi)^d}
  \int_0^t\!\!\!\int_0^s\!\!\!\int_0^s\!\!\!\int_{\R^d}
  {\rm e}^{-(t-s)|\xi|^2} &{\rm e}^{-(s-\sigma)|\xi-\eta|^2\tau^{-1}}{\rm e}^{-(s-\lambda)|\eta|^2\tau^{-1}}
\\
  &\times |\xi-\eta|^2|\eta|^2\widehat u(\xi-\eta,\sigma)\widehat v(\eta,\lambda)
  \dd s\dd\lambda\dd\sigma\dd\eta.
\end{split}
\end{equation} 
When $u$ and $v$ both belong to $\Y_a$ (with $\|u\|_{\Y_a}=\|v\|_{\Y_a}=1$ for simplicity)
we can estimate
\begin{equation}
\begin{split}
|\widehat{B''(u,v)}|(\xi,t)
\le \frac{\tau^{-2}}{(2\pi)^d}
  \int_0^t\!\!\!\int_0^s\!\!\!\int_0^s\!\!\!\int_{\R^d}
  &{\rm e}^{-(t-s)|\xi|^2}{\rm e}^{-(s-\sigma)|\xi-\eta|^2\tau^{-1}}{\rm e}^{-(s-\lambda)|\eta|^2\tau^{-1}}
\\
  &\;\times |\xi-\eta|^{2-a}|\eta|^{2-a}\sigma^{-1+(d-a)/2}\lambda^{-1+(d-a)/2}
  \dd s\dd\lambda\dd\sigma\dd\eta.
\end{split}
\end{equation}
Applying twice~Lemma~\ref{lem:integ} with $A=|\xi-\eta|^2\tau^{-1}$ or
$A=|\eta|^2\tau^{-1}$, and in both cases $\delta=(d-a)/2$,
we obtain, for any $0\le b\le 1$,
\begin{equation}
\begin{split}
|\widehat{B''(u,v)}|(\xi,t)
\le \frac{64\,\tau^{-2+2b}}{(2\pi)^d(d-a)_*^2}
  \int_0^t\!\!\!\int_{\R^d}
  {\rm e}^{-(t-s)|\xi|^2}
  |\xi-\eta|^{2-a-2b}|\eta|^{2-a-2b} s^{d-a-2b}
  \dd s \dd\eta.
\end{split}
\end{equation}
Applying now Lemma~\ref{lem:convo} with $\alpha=\beta=a+2b-2$ we get
\begin{equation*}
\begin{split}
|\widehat{B''(u,v)}|(\xi,t)
\le 
  \frac{64\,C(a+2b-2,a+2b-2,d)\,\tau^{-2+2b}}{(2\pi)^d(d-a)_*^2}  
  |\xi|^{4-2a-4b+d}\int_0^t  {\rm e}^{-(t-s)|\xi|^2}
  s^{-1+(1+d-a-2b)}
  \dd s.
\end{split}
\end{equation*}
The last step is the application of Lemma~\ref{lem:integ} with $A=|\xi|^2$ and 
$\delta=1+d-a-2b$. We find, for any $0\le \gamma\le1$,
\begin{equation}
\begin{split}
|\widehat{B''(u,v)}|(\xi,t)
\le K''(a,b,d)\,\tau^{-2+2b}\,
  |\xi|^{4-2a-4b+d-2\gamma} t^{1+d-a-2b-\gamma},
\end{split}
\end{equation}
where
\begin{equation}
K''(a,b,d)=\frac{256\,C(a+2b-2,a+2b-2,d)}{(2\pi)^d(d-a)_*^2(1+d-a-2b)_*} .
\end{equation}
We want $B''(u,v)$ to belong to $\Y_a$: this requires
$4-2a-4b+d-2\gamma=-a$ and $1+d-a-2b-\gamma=-1+(d-a)/2$.
Both equalities are satisfied when
\[
\gamma=2+\frac{d}{2}-\frac{a}{2}-2b.
\] 
Let us collect all the conditions that we need on $a$ and $b$ for the applicability of the previous lemmas:
we needed $0\le b\le 1$, $d-a>0$ for the first application of Lemma~\ref{lem:integ},
$0<a+2b-2<d$, $2a+4b-4>d$ for the application of Lemma~\ref{lem:convo},
$1+d-a-2b>0$ for the second application of Lemma~\ref{lem:integ} 
and finally 
$0\le 4+d-a-4b\le 2$ (the latter condition correspond to the restriction $0\le \gamma\le1$).
Thus, we can collect all these conditions in the following system
\begin{equation}
\label{sys''}
\begin{cases}
\frac12< b\le 1,\\
\frac{d}{2}-2b+2<a<d+1-2b,\\
2+d-4b\le a<d.
\end{cases}
\end{equation}
Conditions~\eqref{sys''} are not compatible for $d=1,2$. For $d=3$ and $\frac{3}{2}<a<3$,
or otherwise for $d\ge4$ and $d-2\le a<d$, $a\not=2$, we can always find $b$ satisfying
the system.
Therefore, we deduce that
\[
B''\colon\Y_a\times \Y_a \to \Y_a\quad\text{is continuous}
\]
for either 
\[
\begin{cases}
d=3\\
\frac{3}{2}<a<3
\end{cases}
\quad\text{or}
\quad
\begin{cases}
d\ge4\\
d-2\le a<d, \quad a\not=2.
\end{cases}
\]  
For each $\frac12<b\le1$, a suitable choice for $a$ which is 
compatible with system~\eqref{sys''}
is, e.g., 
\[
a=d+\textstyle\frac{4}{3}-\frac{8}{3}b \qquad (d\ge3) 
\]
(for $d\ge4$ we could also choose, e.g., $a=d+\frac32-3b$).  
We have 
\[
K''(d+{\textstyle\frac43-\frac83b,b,d})
= \frac{2^{8-d}\,C(d-\frac23-\frac{2b}{3},d-\frac23-\frac{2b}{3},d)}{\pi^{d}(\frac{8}{3}b-\frac43)_*^2\,\,(\frac23b-\frac13)_*}.
\]
Therefore we get
\[
K''(d+{\textstyle\frac43-\frac83b,b,d})\approx \frac{1}{(2b-1)^3},
\qquad\text{as $\textstyle b\searrow\frac12$}
\]
and that $K''(d+{\textstyle\frac43-\frac83b,b,d})$ remains bounded for $b$ in any interval
of the form $[\frac12+\epsilon,1]$, with $0<\epsilon\le \frac12$.

Let us set $\beta=2b-1$, so that $0<\beta\le 1$.
Then, for $d\ge3$, we can find a constant $\kappa_d''$,   depending only on $d$
such that
\begin{equation}
\|B''(u,v)\|_{\Y_{d-\frac43\beta}}\le \frac{1}{4\kappa''_d}\, \frac{\tau^{-1+\beta}}{\beta^3} 
 \,\|u\|_{\Y_{d-\frac43\beta}}
 \, \|v\|_{\Y_{d-\frac43\beta}},
\end{equation}
for any $0<\beta\le1$.
We can also prove that, still for $0<\beta\le 1$, the operator
\[
B''\colon\Y_{d-\frac43\beta}\times \Y_{d-\frac43\beta}\to \mathcal{X}
\]
is continuous.
Indeed, for this, we just need to choose a different value for $\gamma$, namely
\[
\gamma=1+d-a-2b.
\]
This leads us to put the new condition $0\le 1+d-a-2b\le1$. But this new condition is already ensured by system~\eqref{sys''}.

At this stage, one sees that the existence (and the uniqueness in a ball) of a solution to {\rm (TM')} can be established in the same space as in Theorem~\ref{th:PM} and under the same size condition~\eqref{small-assu}. 
\end{proof}

\section{Finite time blowup for both toy models: an iterative method involving the Fourier transform}

In this section we establish the finite time blowup for solutions to both toy models, assuming that the initial data are large.
For the blowup we will discuss to what extent the size conditions found in the previous section are sharp in the limit $\tau\to\infty$.

The first method leads to finite time blowup for {\rm (TM)}, for a class of initial data
such that $\|u_0\|_{\PM^{d-2}}\gtrsim \tau$, for $\tau$ large enough. 
Therefore, the gap between the global existence and the blowup for such models is only logarithmic in $\tau$.

The second method, a more ``traditional''  eigenfunction or moment one, will be presented in the next section for the system {\rm (TM')} considered in bounded domains with the Dirichlet boundary conditions. 
It seems that other classical approaches to blowup questions in nonlinear parabolic equations (such as the energy method, convexity method, Fujita method), see \cite[Ch. 17]{QS} do not work for those systems because of the structure differences compared to the case of a single equation, and in particular, because of different diffusivities in the both equations.


\subsection{The model {\rm (TM)}}

Let us recall that for, $u_0\in L^{d/2}$, $d\ge3$, there exists a local-in-time solution
to (TM), unique in the class precised in~Remark~\ref{rem:local}. The following theorem
gives a sufficient condition on $u_0$, for such a solution, to blowup in finite time.

Our approach is closely related to that in \cite[Theorem 3.1]{BB1} which was inspired by the blowup result in \cite{M-S} for the so called ``cheap'' Navier--Stokes equations, producing lower bound estimates for the Fourier transform.
The key idea is that the structure of the integral on the right hand side of~\eqref{FTMsol}
 implies that, if $\widehat u_0(\xi)\ge0$, then the positivity of the Fourier transform
will be preserved by the sequence of approximate solutions~$u_k$, and so by the limit $u$ of such sequence as $k\to\infty$ when the latter does exist in an appropriate sense.

Consider $w_0\in L^2(\R^d)$ defined by
\[
\widehat w_0(\xi)=\un_{B_0}(\xi),
\]
where $\un_E$ denotes the indicator function of a measurable set $E$,
and $B_0$ is the ball with center $\frac34(1,0\ldots,0)$ and radius $\frac14$.
Thus, the support of $\widehat w_0$ is contained 
in the annulus 
$
E_0=\{\frac12\le |\cdot|\le 1\}$.

Notice that $w_0$ is a complex valued function, but there exist of course
real valued functions $u_0\in L^2(\R^d)$, and even in $\mathscr{S}(\R^d)$ such that $\widehat u_0\ge \widehat w_0$: for example, one can consider a function proportional to a well-chosen Gaussian.

\begin{theorem}
\label{theoremCC}
Let $d\ge3$, $\tau>0$, $A>0$, and $u_0\in \mathscr{S}(\R^d)$, such that
\[
\widehat u_0(\xi)\ge A\widehat w_0(\xi).
\]
Let $t^*$ be the maximal lifetime of the (unique) solution to (TM).
There exists a constant $\kappa_d>0$ (only dependent on $d$) such that if
\begin{equation}
\label{bucoo}
  A > \kappa_d\,{\rm e}^{1/\tau}\,\tau,
\end{equation}
then  $t^*<1$.
\end{theorem}

Notice that the right-hand side in~\eqref{bucoo} behaves like $\tau$ as $\tau \gg 1$.
Thus, the best possible size condition to be put on the initial data, in order to obtain the global existence for {\rm (TM)}, would be of the form
\[
\|u_0\|\lesssim \tau,
\]
no matter the choice of the norm, and 
irrespectively of  the functional setting where one constructs the solution.

\begin{remark}
\label{rem:moret}
In fact, our argument proves more than this. Indeed, we will actually prove the
following assertion that is stronger than Theorem~\ref{theoremCC}.
Let $\tau>0$, $A>0$, $t^* \ge 1$ and 
\[
\widehat u_0(\xi)\ge A\widehat w_0(\xi).
\]
If $u\in {\mathcal C}((0,T),L^\infty(\R^d))$ is a mild solution to {\rm (TM)} on $(0,T)$ such that $\widehat u\ge0$, and if $A$ is such that
\begin{equation}
\label{buc}
 (3t^*-1+ {\rm e}^{-4t^*})  A > \kappa_d\,{\rm e}^{t^*/\tau}\,{\rm e}^{t^*}\tau,
\end{equation}
then we must have $0<T\le t^*$.
In particular, under condition~\eqref{buc} the $L^\infty(\R^d)$-norm of $u(t)$
must blow up in a finite time.
\end{remark}

\begin{proof}[Proof of Remark~\ref{rem:moret}]
By contradiction, assume $T>t^*$ and let $u\in{\mathcal C}((0,T),L^\infty(\R^d))$ be a solution to {\rm (TM)} with $\widehat u\ge0$.
Let us define, for $k=1,2,\dots$, the dyadic ball
\[
B_k=B_{k-1}+B_{k-1}
\]
and the dyadic annulus
\[
E_k=\left\{\xi\in\R^d\colon 2^{k-1}\le |\xi|\le 2^k\right\}.
\]
Let, for any integer $k\ge1$, $w_k=w_0^{2^k}$, in a such way that $w_k=w_{k-1}^2$.
Then  
\[
\widehat w_k= (2\pi)^{-d}\widehat w_{k-1}*\widehat w_{k-1},
\]
and therefore, 
\[
\hbox{supp\,}\widehat w_k\subset B_k\subset E_k. 
\]

\begin{lemma}
\label{indu}
Under the assumption of Theorem \ref{theoremCC},
for all $k=0,1,2,\ldots$, and $t\in[0,T]$, we have
\begin{equation}
\label{fort}
\widehat u(\xi,t)\ge \beta_k{\rm e}^{-2^k t}\un_{\{t_k\le t<t^*\}}(t)\widehat w_k(\xi),
\end{equation}
where $(\beta_k)$ and $(t_k)$ are two sequences defined below in~\eqref{tk} and \eqref{betak}.
\end{lemma}

\begin{proof}
For $k=0$, the conclusion immediately follows from \eqref{eq:bilftm}. Indeed, the first term gives us
$$\widehat u(\xi,t)\ge A{\rm e}^{-t|\xi|^2}\widehat w_0(\xi),$$
provided we choose $\beta_0=A$ and $t_0=0$.

Let $k\ge1$. Assume that the inequality of the lemma holds with $k-1$ instead of $k$. 
 Then, for all $t_k\le t<t^*$, considering only the bilinear part in \eqref{eq:bilftm}, we get
\begin{equation*}
 \begin{split}
 \widehat u(\xi,t)
&\ge (2\pi)^{-d}\int\limits_{t_{k-1}}^{t}\int\limits_{t_{k-1}}^s\int\limits_{\R^d} 
\frac{|\eta|^2}{\tau}
{\rm e}^{-(t-s)|\xi|^2}{\rm e}^{-\frac{1}{\tau}(s-\sigma)|\eta|^2}
\beta_{k-1}^2 {\rm e}^{-2^{k-1}s}{\rm e}^{-2^{k-1}\sigma} 
\\
&\qquad\qquad\qquad\qquad\qquad\times
\widehat w_{k-1}(\xi-\eta)\widehat w_{k-1}(\eta)\dd\eta\dd\sigma\dd s\\
&
\ge (2\pi)^{-d} \int\limits_{t_{k-1}}^{t}\int\limits_{\R^d} 
(s-t_{k-1})
\frac{2^{2k-4}}{\tau}
{\rm e}^{-(t-s)|\xi|^2}{\rm e}^{-\frac{1}{\tau}(t^*-t_{k-1})2^{2k-2}}
\beta_{k-1}^2 {\rm e}^{-2^k s}
\\
&\qquad\qquad\qquad\qquad\qquad\times
\widehat w_{k-1}(\xi-\eta)\widehat w_{k-1}(\eta)\dd\eta\dd s.\\
\end{split}
\end{equation*}
Thus, we can bound $\widehat u(\xi,t)$ from below as follows
\begin{equation*}
\begin{split}
\widehat u(\xi&,t) \ge \int\limits_{t_{k-1}}^{t} 
(s-t_{k-1})
\frac{2^{2k-4}}{\tau}
{\rm e}^{-(t-s)2^{2k}}{\rm e}^{-\frac{1}{\tau}(t^*-t_{k-1})2^{2k-2}}
\beta_{k-1}^2 {\rm e}^{-2^{k}s}\widehat w_{k}(\xi)\dd s\\
&
\ge 
\biggl(\int\limits_{t_{k-1}}^{t} 
(s-t_{k-1})
{\rm e}^{-(t-s)2^{2k}}
\dd s
\biggr)
\frac{2^{2k-4}}{\tau}
{\rm e}^{-\frac{1}{\tau}(t^*-t_{k-1})2^{2k-2}}
\beta_{k-1}^2 {\rm e}^{-2^k t}\widehat w_{k}(\xi).
\end{split}
\end{equation*}
Integrating by parts then we get, always for $t_k\le t<t^*$,
\begin{equation*}
\begin{split}
\widehat u(\xi,t)
&\ge
\biggl(
\frac{t-t_{k-1}}{2^{2k}}-\frac{1-{\rm e}^{-(t-t_{k-1})2^{2k}}}{2^{4k}}
\biggr)
\frac{2^{2k-4}}{\tau}
{\rm e}^{-\frac{1}{\tau}(t^*-t_{k-1})2^{2k-2}}
\beta_{k-1}^2 {\rm e}^{-2^{k}t}\widehat w_{k}(\xi)\\
&\ge
\biggl(
(t_k-t_{k-1})-\frac{1-{\rm e}^{-(t^*-t_{k-1})2^{2k}}}{2^{2k}}
\biggr)
\frac{2^{-4}}{\tau}
{\rm e}^{-\frac{1}{\tau}(t^*-t_{k-1})2^{2k-2}}
\beta_{k-1}^2 {\rm e}^{-2^{k}t}\widehat w_{k}(\xi).
\end{split}
\end{equation*}
We would like to establish the two bounds from below 
\begin{equation}
\label{needs}
(t_k-t_{k-1})\ge 3t^*2^{-2k}\qquad\text{and}\qquad
{\rm e}^{-\frac{1}{\tau}(t^*-t_{k-1})2^{2k-2}}\ge {\rm e}^{-\delta},
\qquad \text{for all $k\ge1$,}
\end{equation}
for a suitable  positive constant $\delta$ to be chosen later.
This would imply that, for $t_k\le t<t^*$,
\begin{equation*}
 \begin{split}
  \widehat u(\xi,t)
   &\ge
    (3t^*-1+ {\rm e}^{-4\delta\tau}) 2^{-2k-4}\tau^{-1} {\rm e}^{-\delta} \beta_{k-1}^2 {\rm e}^{-2^k\,t} \widehat w_k(\xi).\\
 \end{split}
\end{equation*}
This last inequality is interesting only if the right hand side is positive.
This requires 
\[
3t^*-1+{\rm e}^{-4\delta\tau}>0.
\]

To ensure the validity of the first bound of~\eqref{needs} (and the condition $t_0=0$), a natural choice is
\begin{equation}
 \label{tk}
t_{k}=t^*\Bigl(1-\textstyle\frac{1}{4^{k}}\Bigr),
\end{equation}
so that $t_k\nearrow t^*$.
The second bound of \eqref{needs} can be rewritten as $(t^*-t_k)2^{2k}\le \delta\tau$; 

\noindent 
we can ensure its validity for all $k\in\N$ as soon as $\delta$ is such that
\[
t^*\le \delta\tau.
\]
The above conditions $3t^*-1+{\rm e}^{-4\delta\tau}>0$ and $\delta\tau\ge t^*$
motivate us to restrict ourselves to $t^*\ge 1$.
 
We choose $(\beta_k)$ in  such a way that
\begin{equation*}
\beta_0=A, \qquad \beta_{k}=(3t^*-1+ {\rm e}^{-4\delta\tau})
 2^{-2k-4}\tau^{-1} {\rm e}^{-\delta} \beta_{k-1}^2, \qquad k=1,2,\ldots\ .
\end{equation*}
This choice leads to inequality~\eqref{fort}.
The assertion of the lemma follows by induction.
\end{proof}

In order to compute $\beta_k$, we introduce
$M$, such that
$$
 2^{M}=(3t^*-1+ {\rm e}^{-4\delta\tau}) {\rm e}^{-\delta}\,2^{-4}\tau^{-1}.
$$ 
Notice that we have $\beta_k=2^{M-2k}\beta_{k-1}^2$ for $k\ge1$.
Recalling that $\beta_0=A$, we deduce that 
\begin{equation}
\label{betak}
 \beta_k = \Bigl(A\,2^{M-4}\Bigr)^{2^k}\,2^{4-M+2k}, \qquad k=0,1,\ldots\ .
\end{equation}

The Fourier inversion formula and the positivity of $\widehat u$ 
imply that 
\[
\|u(t)\|_\infty\ge  |u(0,t)|=(2\pi)^{-d}\|\widehat u(t)\|_1.
\]
Moreover, by Lemma~\ref{indu}, for all $k\in\N$,
\[
\|\widehat u(t_k)\|_1
\ge \beta_k{\rm e}^{-2^kt_k}
\|\widehat w_k\|_1
\ge \beta_k {\rm e}^{-2^kt^*}\|\widehat w_k\|_1.
\]
By the fact that $\widehat w_k\ge0$ for all $k\ge0$, Fubini's theorem and
the formula for the volume of the unit ball $\omega_d=\frac{\pi^{d/2}}{\Gamma(1+d/2)}$, we have
\[
\begin{split}
\|\widehat w_k\|_1
&=(2\pi)^{-d}\|\widehat w_{k-1}\|_1^2=
\ldots=
((2\pi)^{-d})^{2^k-1}\|\widehat w_0\|_1^{2^k}\\
&=(2\pi)^d(K_d)^{2^k},
\end{split}
\]
with
\[
K_d=\frac{1}{8^d\pi^{d/2}\,\Gamma(1+d/2)}.
\]
We conclude that when 
\[
A2^{M-4}{\rm e}^{-t^*}K_d>1,
\]
then we have 
$\|u(t_k)\|_{L^\infty}\to\infty$ for $k\to\infty$.
But $t_k\to t^*$ and so
\[
\limsup_{t\to t^*} \|u(t)\|_\infty=\infty.
\]
This is in contradiction with the assumption that
$u\in {\mathcal C}((0,T),L^\infty(\R^d))$ with $T>t^*$.

The size condition on~$A$ above can be rewritten in an equivalent form as
 $$(3t^*-1+ {\rm e}^{-4\delta\tau})  A > \kappa_d{\rm e}^\delta\,\tau\,{\rm e}^{t^*}$$
where $\kappa_d>0$ is an explicit constant depending only on $d$. 
But $\delta \le t^*/\tau$ and so a sufficient blowup condition is 
\[
 (3t^*-1+ {\rm e}^{-4t^*})  A > \kappa_d {\rm e}^{t^*(1+1/\tau)}\,\tau. 
\] 
\end{proof}  

\subsection{The model {\rm (TM')}}

Similarly to the subsection above, we want to prove explosion in finite time for large initial data.
 This time, we obtain a lower bound of the type $\tau^2$ so there is a much larger discrepancy with respect to the upper estimates of the form $\tau(\log \tau)^{-3}$ sufficient for the global existence in Sec. 3. 
Similar lower bounds appear in  Proposition 5.1 for the problem in bounded domains. 

To simplify notation, we assume from the beginning that $\tau \ge 1/2$, and we look for initial data which would ensure explosion at time $t^*=1$. As previously, our goal is to prove by induction that for some $A(\tau)>0$, there is a sequence $\beta_k \rightarrow \infty$ such that for $t_k=1-4^{-k}$, 
provided 
$$
\widehat u_0(\xi)\ge A\widehat w_0(\xi),
$$
we have
\begin{equation}
\label{induc}
\widehat u(\xi,t)\ge \beta_k{\rm e}^{-2^k t}\un_{\{t_k\le t<t^*\}}(t)\widehat w_k(\xi),
\end{equation}
with the $w_k$ defined as above, which implies as previously $\Vert u(t_k) \Vert_{\infty} \rightarrow \infty$.
\\
We begin by using the Fourier representation \eqref{bilfTM'}.
\begin{align*}
u(t)={\rm e}^{t\Delta}u_0+B''(u,u)(t)
\end{align*}
with
\begin{align*}
 \widehat{B''(u,u)}(\xi,t)
= \frac{\tau^{-2}}{(2\pi)^d}
  \int_0^t\!\!\!\int_0^s\!\!\!\int_0^s\!\!\!\int_{\R^d} &
  {\rm e}^{-(t-s)|\xi|^2} {\rm e}^{-(s-\sigma)|\xi-\eta|^2\tau^{-1}}{\rm e}^{-(s-\lambda)|\eta|^2\tau^{-1}}
\\
  &\times |\xi-\eta|^2|\eta|^2\widehat u(\xi-\eta,\sigma)\widehat u(\eta,\lambda)
  \dd s\dd\sigma\dd\lambda\dd\eta.
\end{align*}
Again, we proceed by induction and assuming that~\eqref{induc} holds with $k-1$ instead of~$k$, we deduce that for $t_k \le t <t^*$,
\begin{align*}
\widehat u(\xi,t)
=& (2 \pi)^{-d} \tau^{-2} \int_0^t\!\!\!\int_0^s\!\!\!\int_0^s\!\!\!\int_{\R^d}
  {\rm e}^{-(t-s)|\xi|^2} {\rm e}^{-(s-\sigma)|\xi-\eta|^2\tau^{-1}}{\rm e}^{-(s-\lambda)|\eta|^2\tau^{-1}}
\\
  &\qquad\times |\xi-\eta|^2|\eta|^2\widehat u(\xi-\eta,\sigma)\widehat u(\eta,\lambda)
  \dd s\dd\sigma\dd\lambda\dd\eta 
\\
\ge & (2 \pi)^{-d} \tau^{-2} \int_{t_{k-1}}^t\!\!\!\int_{t_{k-1}}^s\!\!\!\int_{t_{k-1}}^s\!\!\!\int_{\R^d}
  {\rm e}^{-(t-s)|\xi|^2}  {\rm e}^{-(s-\sigma)|\xi-\eta|^2\tau^{-1}} |\xi-\eta|^2\widehat u(\xi-\eta,\sigma)
\\
  &\qquad\times
	{\rm e}^{-(s-\lambda)|\eta|^2\tau^{-1}} |\eta|^2 \widehat u(\eta,\lambda)
  \dd s\dd\sigma\dd\lambda\dd\eta
\\
\ge& (2 \pi)^{-d} \tau^{-2} \int_{t_{k-1}}^t\!\!\!\int_{t_{k-1}}^s\!\!\!\int_{t_{k-1}}^s\!\!\!\int_{\R^d}
  {\rm e}^{-(t-s)|\xi|^2}   {\rm e}^{-2^{2k-2} (s-\sigma)/\tau} 2^{2k-4} \beta_{k-1} \widehat w_{k-1}(\xi-\eta) e^{-2^{k-1} t}
\\	
&	\qquad\times {\rm e}^{-2^{2k-2} (s-\lambda)/\tau} 2^{2k-4} \beta_{k-1} \widehat w_{k-1}(\eta)
	e^{-2^{k-1} t} \dd\eta\dd\sigma\dd\lambda\dd s
\\
\ge&	2^{4k-8} \tau^{-2} {\rm e}^{-2^{2k-1} (1-t_{k-1})/\tau} \int_{t_{k-1}}^t\ (s-t_{k-1})^2 {\rm e}^{(s-t)2^{2k}} ds
\times \beta_{k-1}^2 e^{-2^k t} \widehat w_k(\xi)
\\
\ge& 2^{4k-8} \tau^{-2} {\rm e}^{-2^{2k} (1+1/2\tau)(1-t_{k-1})} \int_{t_{k-1}}^t\ (s-t_{k-1})^2 ds
\times \beta_{k-1}^2 e^{-2^k t} \widehat w_k(\xi)
\\
\ge& \frac{2^{4k-8}}{3}	\times \tau^{-2} {\rm e}^{-2^{2k} (1+1/2\tau)(1-t_{k-1})} (t-t_{k-1})^3 
\beta_{k-1}^2 e^{-2^k t} \widehat w_k(\xi)
\\
\ge& \frac{2^{4k-8}}{3}	\times \tau^{-2} {\rm e}^{-2^{2k+1} (1-t_{k-1})} (t_k-t_{k-1})^3 
\beta_{k-1}^2 e^{-2^k t} \widehat w_k(\xi). 
\end{align*}
In the last inequality, we used that $\tau \ge 2$. Now, with the same $t_k$ as in \eqref{tk} (we recall that here $t^*=1$), we get  for some $K>0$ and for $t_k \le t <t^*$:
$$
\widehat u(\xi,t) \ge (K \tau^2)^{-1} 2^{-2k} {\rm e}^{-2^{k}t} \beta_{k-1}^2 \widehat w_k(\xi).
$$
So if we define
$$
\beta_0:=A;\ \beta_k:=(K \tau^2)^{-1} 2^{-2k} \beta_{k-1}^2,
$$
we have proved that
$$
\widehat u(\xi,t)\ge \beta_k{\rm e}^{-2^k t}\un_{\{t_k\le t<t^*\}}(t)\widehat w_k(\xi).
$$
As previously, we obtain that
$$
\beta_k=\Big( A/ 16 K\tau^2 \Big)^{2^k} 16 K \tau^{2} 2^{2k},
$$
which implies as above that we have blow up provided
$$
A \ge 16 C_d K \tau^2,
$$
with $C_d>0$ only depending on the dimension.

\section{Blowup for the toy model {\rm (TM')} in bounded domains}


The toy model {\rm (TM')} features cross-diffusion terms which makes its existence and regularity of solutions delicate as we have already seen. Compared to the standard example of the square nonlinear heat equation {\rm (NLH)},  
the nonlinearity in {\rm (TM')}  is weaker that that in {\rm (NLH)}. 
Indeed, one may represent 
\begin{equation}\label{Lapv}
\Delta \varphi(t)=\frac1\tau \Delta{\rm e}^{\frac{t}{\tau}\Delta}\varphi_0 +\frac1\tau\int_0^t \Delta{\rm e}^{\frac{t-s}{\tau}\Delta}u(s)\dd s.
\end{equation}  
Now it is clear that $-\Delta \varphi(t)$ is somewhat smoother and smaller than $\frac1\tau u(t)$ itself since \eqref{Lapv} involves  two smoothing (or rather ``averaging'') operators with respect to the both variables $t$ and $x$, namely $\int_0^t \dots \dd s$ and $\Delta{\rm e}^{t\Delta}$. Thus, one expects  sufficient conditions for blowup of {\rm (TM')} to be stronger with $\tau\gg 0$ than those for $\tau=0$, i.e. for equation {\rm (NLH)}. 
Note that blowup results for parabolic systems, in particular for various versions of doubly parabolic systems of chemotaxis, are scarce since  there are  no   general comparison principles available for parabolic systems, even for linear ones. For some examples for the Keller--Segel system see, e.g., \cite{Wi1,Wi2,Wi3}, where supplementary information on solutions is extracted from entropy functionals and other specific properties of those drift-diffusion systems. 

Here, for simplicity of presentation, let us consider {\rm (TM')} in a smooth bounded domain $\Omega\subset {\mathbb R}^d$, supplemented with the homogeneous Dirichlet conditions for both $u$ and $\varphi$
\begin{equation}\label{D}
u(x,t)=\varphi(x,t)=0\ \ \ \ {\rm for\ \ each\ \ \ } x\in\partial\Omega,\ \ \ t\ge 0.
\end{equation}

\begin{proposition}\label{blowTM2}
There are no global-in-time classical solutions of system (TM') with $\tau\ge 2$ and suitably large $u_0\ge 0$, $\varphi_0\ge 0$ (of order $\tau^2$, $\tau$, respectively). 
More precisely, if 
\begin{equation}\label{contr2}
\int_\Omega \psi(x)u_0(x)\dx\ge \frac32\lambda \tau^2,
  \int_\Omega \psi(x)\varphi_0(x)\dx\ge  \frac32\tau,
\end{equation}
where $\psi\ge 0$ is the normalized eigenfunction of $\Delta$ with the first  eigenvalue $\lambda$, then classical solutions of system {\rm (TM')}  cannot be defined globally in time.

\end{proposition}

The proof of that result involves a generalization of the classical construction  for nonlinear heat equations, see e.g. \cite[Th. 17.1]{QS}, namely the moment method invented about 50 years ago by S. Kaplan. This can be applied to sufficiently regular weak solutions. Indeed, a smooth function vanishing at the boundary  ($\psi$ below) can be used as a test function, and time-continuity in the weak sense is needed to perform the computations below. 

\begin{proof}
Let $\psi$ be the first eigenfunction of the Dirichlet Laplacian on the domain $\Omega$ 
$$\Delta \psi+\lambda \psi=0,\ \ \ \psi(x)=0\ \ \ \ {\rm for\ \ each\ \ \ } x\in\partial\Omega,$$ 
which may be chosen so that $\psi(x)>0$ for each 
$x\in\Omega$ and $\int_\Omega \psi(x)\dd x=1$, here $\lambda>0$ is the first eigenvalue of the Dirichlet Laplacian.  
Define the moments
\begin{equation}\label{J}
J(t)=\int_\Omega \psi(x)\varphi(x,t)\dd x,\ \ \ I(t) =\int_\Omega \psi(x)u(x,t)\dd x.
\end{equation} 
Solutions of the system {\rm (TM')} satisfy  a single, second order in time equation for $\varphi$  
\begin{equation}\label{eq-phi}
\tau \varphi_{tt}=(\tau+1)\Delta\varphi_t-\Delta^2\varphi+(\Delta\varphi)^2.
\end{equation}

 Suppose that we have a global-in-time solution $\varphi$ of \eqref{eq-phi}. 
It is easy to see that after some integrations by parts we obtain
\begin{equation}
\begin{aligned}\label{diffineq}
\tau\ddot J(t)&=-\lambda(\tau +1)\dot J(t) -\lambda^2 J(t)+\int_\Omega \psi(\Delta\varphi)^2 \\
&\ge -\lambda(\tau +1)\dot J(t) -\lambda^2 J(t)+ {\lambda^2} J(t)^2 .
\end{aligned}
\end{equation} 
Indeed, one has by the Jensen inequality 
$$\int_\Omega \psi(\Delta\varphi)^2\ge\left(\int_\Omega \psi|\Delta\varphi|\right)^2\ge \left(\int_\Omega \psi \Delta\varphi \right)^2=\left(\int_\Omega \Delta\psi\, \varphi\right)^2 =\lambda^2J^2.$$ 
Now, the differential inequality \eqref{diffineq} reads 
\begin{equation}\label{diffineq2}
\ddot J(t)+\lambda\left(1+\frac1\tau\right)\dot J(t)+\frac{\lambda^2}{\tau}J(t)\ge \frac{\lambda^2}{\tau}J(t)^2.
\end{equation} 
Taking $X(t)={\rm e}^{\alpha t}J(t)$, i.e. ``zooming in''   $J(t)$, with $\alpha=\frac\lambda{2}\left(1+\frac1\tau\right)$, we arrive at 
\begin{equation}\label{X}
\ddot X(t)-\frac{\lambda^2}{4}\left(1 - \frac1\tau\right)^2X(t)\ge \frac{\lambda^2}{\tau}{\rm e}^{-\alpha t}X(t)^2. 
\end{equation} 
Clearly, $X(t)\ge 0$ but also $\dot X(t)\ge0$ since 
$\sgn \dot X=\sgn \left(\dot J +\alpha J  \right).$ 
Indeed, we have $\dot J =-\frac{\lambda}{\tau}J +\frac1\tau I $, so that 
\begin{equation}\label{XJI}
\dot J+\alpha J=\frac\lambda{2}\left(1-\frac1\tau\right)J+\frac1\tau I\ge 0
\end{equation}
 for all $\tau\ge 1$. 
In particular, $X(t)>0$ holds if $0\le \varphi\not\equiv 0$. 

Therefore we obtain $2\ddot X(t)\dot X \ge 2\frac{\lambda^2}{\tau}{\rm e}^{-\alpha t}X^2\dot X \ge \frac23\frac{\lambda^2}{\tau} d({\rm e}^{-\alpha t}X(t)^3)/dt$, so we get that  
$$\left(\dot X(t)\right)^2-\left(\dot X(0)\right)^2\ge \frac23\frac{\lambda^2}{\tau}\left({\rm e}^{-\alpha t}X(t)^3-X(0)^3\right).$$ 
Suppose that we can guarantee  that 
\begin{equation}\label{contr4}\left(\dot X(0)\right)^2- \frac23\frac{\lambda^2}{\tau}X(0)^3\ge 0.
\end{equation}  
Under this condition we continue with  
$$\frac{\dot X(t)}{X(t)^{3/2}}\ge  \left(\frac23\frac{\lambda^2}{\tau}\right)^{1/2}{\rm e}^{-\frac{\alpha}{2}t},$$ 
and after an integration 
   $$2\left(\frac{1}{X(0)^{1/2}}-\frac{1}{X(t)^{1/2}}\right)   \ge \left(\frac{2}{3\tau}\right)^{1/2}\lambda\frac{2}{\alpha}\left(1-{\rm e}^{-\frac{\alpha}{2}t}\right).$$
Rearranging the terms we get 
\begin{equation}\label{contr}
\frac{1}{X(t)^{1/2}}\le \frac{1}{X(0)^{1/2}}+ \left(\frac{2}{3\tau}\right)^{1/2}\frac{2}{1+\frac1\tau}\left({\rm e}^{-\frac{\alpha}{2}t}-1\right),
\end{equation}
as long as the solution exists.  
Now, if $\tau\ge 1$ and 
\begin{equation}\label{contr3}
X(0)=J(0)>\frac32\tau, 
\end{equation}
 then we arrive at a contradiction for large times.

From  \eqref{XJI} and the inequality between the arithmetic and geometric means we obtain 
$$\left(\dot X(0)\right)^2\ge 4\frac\lambda{2}\left(1-\frac1\tau\right)J(0)\frac1\tau I(0).$$ 
Thus, to get \eqref{contr4}, it suffices for $\tau\ge 2$ to have $I(0)\ge \frac23\lambda J(0)^2$ since $J(0)=X(0)$. 
Now, conditions \eqref{contr4} and \eqref{contr3}, sufficient for a finite time blowup, for $\tau\ge 2$ follow from
$$J(0)\ge \frac32\tau \ \ \ {\rm and}\ \ \ I(0)\ge \frac23 \lambda J(0)^2.$$
Thus, the condition \eqref{contr2} in Proposition \ref{blowTM2} can be satisfied for suitably chosen $J(0)$ and $I(0)$  of order $\tau$ and $\tau^2$, respectively. 

Summarizing, with sufficiently large $\varphi_0$ ($\varphi_0\gtrsim\tau$) and suitably  large $u_0$ ($u_0\gtrsim \tau^2$) blowup of solutions occurs at a finite time. 
\end{proof}


\medskip

\begin{remark}\label{whole-space}
An immediate generalization of the proof of Proposition \ref{blowTM2} to the case of the whole space $\mathbb R^d$ seems not possible since the relation $-\Delta\psi\ge \lambda\psi$ for some $\lambda>0$ cannot be valid for a~function $0\neq \psi\ge 0$ where, for instance,  $\psi$ has either compact support or decays fast enough as $x\to\infty$, and is sufficiently smooth to integrate $\Delta^2$ by parts four times. The above inequality fails near the boundary of the support of $\psi$. 
\end{remark}


\section*{Acknowledgments}

The authors thank 
Lucilla Corrias, Pierre-Gilles Lemari\'e, Miko{\l}aj Sier\.z\c{e}ga, Philippe Souplet  for interesting conversations on the topic of that work. 
\medskip \\
The first named author would like to thank Institut Camille Jordan, Universit\'e Claude Bernard-Lyon~1 for hospitality during his sabbatical stay (Sep 2021--Jan 2022) as a fellow of {\em Institut des \'Etudes Avanc\'ees -- Collegium de Lyon}, partially supported by  the Polish NCN grant \hbox{2016/23/B/ST1/00434}. 



\end{document}